\documentclass[11pt]{article}
\usepackage{amsmath}
\usepackage{amssymb}
\usepackage{latexsym}
\usepackage{graphicx}
\usepackage{cite}
\usepackage{mathabx}

\usepackage[cp866]{inputenc}
\usepackage[english]{babel}

\usepackage{amsthm, amstext}
\usepackage{array, amsfonts, mathrsfs}

\theoremstyle{plain}
\newtheorem*{thm}{Theorem}
\newtheorem{prop}{Proposition}

\newcommand{\re}[1]{(\ref{#1})}
\newcommand{\diff}[2]{\frac{\partial#1}{\partial#2}}
\newcommand{\diffd}[2]{\frac{\partial^2#1}{\partial#2^2}}
\def\({\left(}
\def\){\right)}

\title{Regularization of an Ill-posed Cauchy Problem for the Wave Equation (Fourier Method)}
\author{M.N. Demchenko\footnote{St.Petersburg Department of V.A.~Steklov Institute of Mathematics of Russian Academy of Sciences.
demchenko@pdmi.ras.ru.
The research is supported by the grants RFBR 15-31-20600-mol-a-ved and RFBR 14-01-00535-a.}}

\date{}

\begin{document}

\maketitle

\begin{abstract}
An ill-posed Cauchy problem for the wave equation is considered:
the solution is to be determined by the Cauchy data
on some part of the time-space boundary.
By means of Fourier method we obtain a regularization algorithm
for this problem, which is given by
rather explicit formula.

\smallskip

\end{abstract}

\medskip

\section{Problem statement} 
%\footnotetext{The research is supported by the grants RFBR 15-31-20600-mol-a-ved and RFBR 14-01-00535-a.}
Suppose $u(x,y,t)$ is a smooth function in $x,t\in{\mathbb R}$, $y \geqslant 0$, satisfying the wave equation and the Dirichlet boundary condition:
\begin{align}
  &\diffd{u}{t} - \diffd{u}{x} - \diffd{u}{y} = 0, \label{waveq} \\
  &u|_{y=0} = 0. \label{bnd}
\end{align}
We consider an ill-posed Cauchy problem for the equation~\re{waveq}: the function $u$ is to be determined by the normal derivative $\diff{u}{y}$ given on the part of the boundary
$\{y=0\}$.
We shall obtain a regularization procedure for determining $u(x_0,y_0,t_0)$, $y_0>0$, which requires the data $\diff{u}{y}$ on the set
\begin{equation}\label{set}
  U := \left\{(x,0,t)\,|\,\, |x-x_0| \leqslant D\(\sqrt{y_0^2-(t-t_0)^2}\)+\varepsilon, |t-t_0| \leqslant y_0 \right\}.
  %U := \left\{ \leqslant \(\) \varepsilon\right\}. %(x,0,t)\,|\,\, |x-x_0| \leq D\(\sqrt{y_0^2-(t-t_0)^2}\)+\e, |t-t_0| \leq y_0 \right\}.
\end{equation}
Here
\[
  D(z) := z \cdot \sqrt{\frac{c}{c + 2z}}, \quad z \geqslant 0,
\]
and $c,\varepsilon$ are arbitrary positive numbers.
Note that the function $D(z)$ increases, so $U$ is a subset of the following rectangle
\begin{equation}\label{rect}
  \left\{(x,0,t)\,|\,\, |x-x_0| \leqslant D(y_0)+\varepsilon, |t-t_0| \leqslant y_0 \right\}.
\end{equation}
The main result of the paper is the following
\begin{thm}
Suppose $y_0>0$, $x_0, t_0 \in{\mathbb R}$.
For a $C^\infty$-smooth solution of the equation~\re{waveq}
in the domain $\{y\geqslant 0\}$ with the boundary condition~\re{bnd}
the following relation holds true
  \begin{equation}\label{main}
    u(x_0, y_0, t_0) = \lim_{h\to 0+} \int_U dx dt\, 
    K_h(x-x_0, y_0, t-t_0)\, \diff{u}{y}(x, 0, t),
  \end{equation}
  where the 
  kernel $K_h$ is defined as follows
\begin{align}
  &K_h(x, y_0, t) :=\notag\\
  &=\frac{1}{2\pi^{3/2}\sqrt{h}} \, {\rm Re}\left[\frac{1}{\sqrt{c+ix}}
  \int_0^{\pi/2} ds\,
  \exp\left(-\frac{1}{4h(c+ix)}\left(x+i\sqrt{y_0^2-t^2}\, 
    {\sin}s\right)^2\right)\right]
  \label{Kh}
\end{align}
(we chose a leaf of the square root $\sqrt{c+ix}$ in such a way that ${\rm Re}\sqrt{c+ix} > 0$).
\end{thm}
Note that the regularization kernel $K_h$ is real and even in both variables $x$ and $t$.

In formula~\re{main} the set of integration $U$ depends on $\varepsilon$,
while the integrand does not. 
Besides,
$K_h(x-x_0,y_0,t-t_0)\to 0$ as $h\to 0$, if
$|x-x_0| > D\(\sqrt{y_0^2-(t-t_0)^2}\)$,
consequently the limit in the r.h.s. of~\re{main} does not depend on $\varepsilon$ (this is discussed in sec.~\ref{formula_derive} in more details).
However, the rate of convergence of the limit 
does depend on the choice of $\varepsilon$. 

The Cauchy problem considered here was solved by R.~Courant~\cite{K};
the reduction to the problem of
recovering function by its spherical means was used.
According to~\cite{K} the Cauchy data on the rectangle
\begin{equation*}
  \left\{(x,0,t)\,|\,\, |x-x_0| \leqslant \delta, |t-t_0| \leqslant y_0 \right\}
\end{equation*}
with arbitrarily small $\delta$ are sufficient to determine $u(x_0,y_0,t_0)$.
These data would be sufficient if one uses the formula~\re{main} as well: 
parameters $c$, $\varepsilon$ can be chosen arbitrarily small
and $D(y_0)\to 0$ as $c\to 0$, so the rectangle~\re{rect}
together with the set $U$ can have arbitrarily small 
size along the $x$-axis.

In~\cite{K} there is no parameter that is analogous to $c$ in~\re{main};
in fact the Cauchy data in the ``infinitesimal'' neighborhood
of the interval $\{x=x_0, |t-t_0|\leqslant y_0\}$ is used
(the derivatives of data and intermediate functions on the interval are calculated).
At the same time for a fixed $c$ the kernel
$K_h(x-x_0,y_0,t-t_0)$ does not tend to zero 
as
$h\to 0$ if $|x-x_0| \leqslant D\(\sqrt{y_0^2-(t-t_0)^2}\)$ (from sec.~\ref{formula_derive} one can conclude that
$K_h$ grows exponentially).
Hence in the r.h.s. of~\re{main} the Cauchy data
on a set of positive ($2$-dimensional) measure are taken into account.
The dependence of stability of regularization of ill-posed problems for hyperbolic equations
on the amount of data 
was studied in~\cite{sh}
(the singular value decomposition was applied):
the larger amount of data provides the more stable regularization.

The problem of determining of the solution of the wave equation by the boundary data in specific domains (such as ball, ellipsoid, half-space) and related problems of integral geometry
were considered (besides~\cite{K} mentioned above) in~\cite{rom, Nat, FPR, Halt, Blag, Symes, Pal}.
In papers~\cite{Nat, FPR, Halt, Blag}
the inversion formulas were obtained, which unlike~\re{main},
require the data on the whole boundary and on sufficiently large time interval depending on the diameter of the domain.
In~\cite{Pal} the problem of recovering of the function in
half-plane by its mean values over circles centered at the boundary of half-plane; the microlocal estimate was obtained under assumption that the function is compactly supported.

To obtain~\re{main} we apply Fourier method to the wave equation.
Applying Fourier transform in $x$, we obtain the Cauchy problem for the wave equation
in the domain $y\geqslant 0, t\in{\mathbb R}$.
The inverse Fourier transform requires a regularization
(in~\re{main} $h$ is a small parameter of regularization).
Note that our problem is a particular case 
of the problem of integral geometry
considered in~\cite{rom},
where also Fourier method was applied.\footnote{In~\cite[Ch. I]{rom} the function in a layer is recovered by its mean values over some family of surfaces;
the latter was supposed to be invariant with respect to translation along transversal directions.
Our problem reduces to recovering a function by its mean values over circles centered at the line $\{y=0\}$.}
Our goal is to obtain a regularization that requires
the Cauchy data only in $U$.

Now we make some obvious simplifications of our problem.
Further we suppose that $x_0=t_0=0$. Thus formula~\re{main} takes the following form
\begin{equation}\label{mainn}
  u(0, y_0, 0) = \lim_{h\to 0+} \int_U dx dt\, 
  K_h(x, y_0, t)\, \diff{u}{y}(x, 0, t).
\end{equation}
It is sufficient to prove formula~\re{mainn} for even function $u$ in $t$.
Indeed, in general case we may consider an even function
\[
  u'(x, y, t) = \left(u(x, y, t) + u(x, y, -t)\right)/2,
\]
which satisfies~\re{waveq}, \re{bnd}, and apply~\re{mainn} to $u'$.
We have $u(0,y_0,0) = u'(0,y_0,0)$.
Since the kernel $K_h$ is even in $t$, the derivative $\diff{(u-u')}{y}$ is odd in $t$,
and the set $U$ is symmetric with respect to $\{t=0\}$,
we have
\begin{align*}
  \int_U dx dt\,
  K_h(x, y_0, t)\, \left(\diff{u}{y}(x, 0, t) - \diff{u'}{y}(x, 0, t)\right) = 0.
\end{align*}
This implies~\re{mainn} in general case.

Further we suppose $u$ to be even in $t$.
From the wave equation~\re{waveq} and boundary condition~\re{bnd} immediately follows
the relation $\partial^n u/\partial y^n = 0$ for $y=0$ and even $n$.
Therefore, an odd continuation of $u$ in $y$ belongs to $C^\infty({\mathbb R}^3)$.
We use the same notation $u$ for such a continuation.

Put $u_0 := u|_{t=0}$, $u_0\in C^\infty({\mathbb R}^2)$.
As $u$ is even in $t$ we have $\partial u/\partial t = 0$ for $t=0$.
Let $\chi(x,y)$ be a $C^\infty$-smooth compactly supported function in ${\mathbb R}^2$ satisfying $\chi(x,y)=1$ for $x^2+y^2 \leqslant R^2$ for some $R$.
Then the solution $u_\chi$ of the problem
\begin{align*}
  &\diffd{u_\chi}{t} - \diffd{u_\chi}{x} - \diffd{u_\chi}{y} = 0, \quad
  u_\chi|_{t=0} = \chi u_0, \quad \diff{u_\chi}{t}\bigg|_{t=0} = 0,
\end{align*}
coincides with $u$ on the set
\[
  \{ (x,y,t)\in{\mathbb R}^3 \,|\, \sqrt{x^2+y^2} + |t| \leqslant R\}.
\]
Choosing sufficiently large $R$ we guarantee that $u_\chi(0,y_0,0) = u(0,y_0,0)$
and
\[
  \diff{u_\chi}{y} = \diff{u}{y}
\]
on the set~\re{set}.
Thus, in proof of formula~\re{mainn} we may suppose that $u_0$ is compactly supported.

\section{Fourier transform of the solution $u$}
Here we study some properties of Fourier transform $\hat u(k,l,\omega)$ of the function
$u(x,y,t)$.
We use the following formulas for Fourier transform (and its inverse) of
function $f(x)$:
\[
  \hat f(k) = \int_{-\infty}^\infty dx\, e^{-ikx} f(x),\quad
  f(x) = \frac{1}{2\pi}\int_{-\infty}^\infty dk\, e^{ikx} \hat f(k),\quad
\]

Since $u_0(x,y)$ is smooth and compactly supported its Fourier transform
$\hat u_0(k,l)$ belongs to $S({\mathbb R}^2)$ (Schwartz space).
Further we use the following estimates
($N$ is a positive integer):
\begin{equation}
  |\hat u_0(k,l)|,\, \left|\diff{\hat u_0(k,l)}{k}\right|,\,
  \left|\diff{\hat u_0(k,l)}{l}\right| \leqslant 
  \frac{C_N}{(1+|k|^N)(1+|l|^N)}. 
  \label{u0decay}
\end{equation}

The function $\hat u$ belongs to $S'({\mathbb R}^3)$ (tempered distributions).
\begin{prop}
  The distribution $\hat u$ acts on the test function
  $\varphi\in S({\mathbb R}^3)$ in the following way
  \begin{equation}\label{hatu}
    \langle  \hat u, \varphi \rangle  = \int_{{\mathbb R}^2} dk dl\, \hat u_0(k,l)\, \frac{1}{2}\left(\varphi\left(k,l,\sqrt{k^2+l^2}\right)+
    \varphi\left(k,l,-\sqrt{k^2+l^2}\right)\right).
  \end{equation}
\end{prop}
\begin{proof}
The r.h.s. of~\re{hatu} defines some distribution $\psi$ from $S'({\mathbb R}^3)$.
Now we show that the inverse Fourier transform $\widecheck\psi$ is a regular function
and the following relation holds true
\begin{equation}\label{checkpsi}
  \widecheck\psi(x,y,t) = \frac{1}{(2\pi)^3} \int_{{\mathbb R}^2} dk dl\, \hat u_0(k,l)
  \frac{1}{2}\left(e^{i(kx + ly +\sqrt{k^2+l^2} \cdot t)} + e^{i(kx + ly -\sqrt{k^2+l^2} \cdot t)}\right).
\end{equation}
(The integral in the r.h.s. absolutely converges due to~\re{u0decay}).
For a test function $\zeta\in S({\mathbb R}^3)$ we have
\begin{align*}
  &\langle \widecheck\psi,\zeta\rangle  = \langle \psi,\widecheck\zeta\rangle  = \frac{1}{(2\pi)^3} \int_{{\mathbb R}^2} dk dl\,
  \hat u_0(k,l) \int_{{\mathbb R}^3} dx dy dt\, \zeta(x,y,t) \frac{1}{2} 
  \sum_\pm e^{i(kx + ly \pm\sqrt{k^2+l^2} \cdot t)} =\\
  &=\frac{1}{(2\pi)^3} \int_{{\mathbb R}^3} dx dy dt\, \zeta(x,y,t)
  \int_{{\mathbb R}^2} dk dl\, \hat u_0(k,l) \frac{1}{2} 
  \sum_\pm e^{i(kx + ly \pm\sqrt{k^2+l^2} \cdot t)}
\end{align*}
($\widecheck\zeta$ is the inverse Fourier transform of $\zeta$).
This implies~\re{checkpsi}.
It can be easily derived from the formula~\re{checkpsi} that $\widecheck\psi$ is the solution
of the following Cauchy problem
\begin{align*}
  &\diffd{\widecheck\psi}{t} - \diffd{\widecheck\psi}{x} - \diffd{\widecheck\psi}{y} = 0, \quad
  \widecheck\psi|_{t=0} = u_0, \quad \diff{\widecheck\psi}{t}\bigg|_{t=0} = 0,
\end{align*}
therefore, $\widecheck\psi = u$, and so $\psi = \hat u$.
\end{proof}

Let $\tilde u(k,y,\omega)$ be Fourier transform of $u(x,y,t)$ in $x$, $t$.
The function $\tilde u(\cdot, y, \cdot)$ belongs to $S'({\mathbb R}^2)$.
\begin{prop}
  For any $y_0$ the function $\tilde u(\cdot, y_0, \cdot)$ is regular and the following relation holds true
  \begin{equation}\label{utif}
    \tilde u(k, y_0,\omega) = 
      \theta(|\omega|-|k|)\,
      \frac{|\omega|}{4\pi\sqrt{\omega^2-k^2}}
      \sum_\pm \hat u_0\(k, \pm\sqrt{\omega^2-k^2}\)\,
      e^{\pm i y_0 \sqrt{\omega^2-k^2}}
  \end{equation}
($\theta$ is the Heaviside function).
\end{prop}
\begin{proof}
For a test function $\varphi(k,\omega)$ belonging to $S({\mathbb R}^2)$ we have
\begin{equation}\label{utilim}
  \langle \tilde u(\cdot, y_0, \cdot), \varphi\rangle  = \langle  u(\cdot, y_0, \cdot), \hat\varphi\rangle  =
  \lim_{\varepsilon\to 0}\, \langle u, f_\varepsilon\rangle , 
\end{equation}
where
\[
  f_\varepsilon(x,y,t) := \hat\varphi(x,t)\, g_\varepsilon(y), \quad
  g_\varepsilon(y) := \frac{e^{-(y-y_0)^2/\varepsilon}}{\sqrt{\pi\varepsilon}},
\]
($g_\varepsilon$ tends to $\delta(y-y_0)$ as $\varepsilon\to 0$). 
Due to~\re{hatu} we have
\begin{align*}
  \langle u, f_\varepsilon\rangle  = \langle \hat u, \widecheck f_\varepsilon\rangle  =
  \int_{{\mathbb R}^2} dk dl\, \hat u_0(k,l)\, \frac{1}{2}
  \sum_\pm \widecheck f_\varepsilon\(k,l,\pm\sqrt{k^2+l^2}\).
\end{align*}
It is easy to see that
\[
  \widecheck f_\varepsilon(k,l,\omega) \to \frac{1}{2\pi}\, e^{i y_0 l} \varphi(k,\omega), \quad \varepsilon\to 0.
\]
Together with~\re{utilim} this yields
\begin{equation}\label{utiphi}
  \langle \tilde u(\cdot, y_0, \cdot), \varphi\rangle  = 
  \int_{{\mathbb R}^2} dk dl\, \hat u_0(k,l)\, \frac{1}{4\pi}
  \sum_\pm e^{i y_0 l} \varphi(k,\pm\sqrt{k^2+l^2}). 
\end{equation}
Consider for example the term ``+'' of sum in the integral in~\re{utiphi}.
Represent its integral as the following sum 
\begin{align*}
  &\int_{{\mathbb R}^2\cap\{l > 0\}} + \int_{{\mathbb R}^2\cap\{l < 0\}}
  dk dl\, \hat u_0(k,l)\, 
  e^{i y_0 l} \varphi(k,\sqrt{k^2+l^2}).
\end{align*}
Now we make change of variables in both integrals
\[
  (k,l) \mapsto (k,\omega), \quad \omega = \sqrt{k^2+l^2}.
\]
We obtain
\begin{align*}
  &\int_0^\infty d\omega \int_{-\omega}^{\omega} dk\, \hat u_0(k,\sqrt{\omega^2-k^2})\,
  e^{i y_0 \sqrt{\omega^2-k^2}} \varphi(k,\omega) \frac{\omega}{\sqrt{\omega^2-k^2}}\, +\\
  &+ \int_0^\infty d\omega \int_{-\omega}^{\omega} dk\, \hat u_0(k,-\sqrt{\omega^2-k^2})\,
  e^{-i y_0 \sqrt{\omega^2-k^2}} \varphi(k,\omega) \frac{\omega}{\sqrt{\omega^2-k^2}}.
\end{align*}
Carrying out analogous calculations for the term ``--'' in the integral~\re{utiphi}
we arrive at~\re{utif}.
\end{proof}

Since the function $u_0(x,y)$ is odd in $y$, its Fourier transform $\hat u_0(k,l)$
is odd in $l$. Hence the formula~\re{utif} can be written as follows
\begin{equation}\label{utifodd}
  \tilde u(k,y_0,\omega) = 
    \theta(|\omega|-|k|)\,
    \frac{i |\omega|}{2\pi \sqrt{\omega^2-k^2}}\,
    \hat u_0\(k, \sqrt{\omega^2-k^2}\)\,
    \sin\(y_0 \sqrt{\omega^2-k^2}\).
\end{equation}

Put
\begin{equation}\label{vdef}
  v(x,t) := \diff{u}{y}(x,0,t)
\end{equation}
and denote by $\tilde v(k,\omega)$ Fourier transform of $v(x,t)$.
It follows from~\re{utifodd} that
\begin{equation}\label{tiv}
  \tilde v(k,\omega) = \frac{i |\omega|}{2\pi}\, \theta(|\omega|-|k|)\,
  \hat u_0\left(k, \sqrt{\omega^2-k^2}\right).
\end{equation}
Now we recast formula~\re{utifodd} as follows
\begin{equation}\label{utifin}
  \tilde u(k,y_0,\omega) = 
  \tilde v(k,\omega)\, \frac{\sin\(y_0 \sqrt{\omega^2-k^2}\)}{\sqrt{\omega^2-k^2}}.
\end{equation}

Relation~\re{tiv} implies that $\tilde v(k, \cdot) \in L_1({\mathbb R})$ for any $k$ 
and the estimate
\begin{equation}\label{tivnorm}
  \int_{-\infty}^\infty d\omega\, |\tilde v(k,\omega)| \le \frac{C}{1+k^2}
\end{equation}
holds true, where $C$ does not depend on $k$.
Indeed,
\begin{align*}
  &\int_{-\infty}^\infty d\omega\, |\tilde v(k,\omega)| =
  \frac{1}{\pi} \int_{|k|}^\infty d\omega\, \omega
  \left|\hat u_0\left(k, \sqrt{\omega^2-k^2}\right)\right| =
  \frac{1}{\pi} \int_0^\infty dl\, l\, |\hat u_0(k, l)|.
\end{align*}
Here the first equality holds true since $u$ is even in $t$ and so $\tilde v$ is even in $\omega$.
In the second equality we made change of variable $l = \sqrt{\omega^2-k^2}$.
Now applying the estimate~\re{u0decay} we arrive at~\re{tivnorm}.

The estimate~\re{tivnorm} means that $\tilde v\in L_1({\mathbb R}^2)$.
Taking into account~\re{utifin} and that $\tilde u$ and $\tilde v$
are supported in the set $\{|\omega|\geqslant |k|\}$, we have
$\tilde u(\cdot, y_0, \cdot)\in L_1({\mathbb R}^2)$.
Applying inverse Fourier transform to $\tilde u$ we obtain
\begin{equation}\label{vsin}
  u(0, y_0, 0) = \frac{1}{4\pi^2} \int_{{\mathbb R}^2} dk\, d\omega\, 
  \frac{\sin\(y_0 \sqrt{\omega^2-k^2}\)}{\sqrt{\omega^2-k^2}}\, \tilde v(k,\omega).
\end{equation}

Further we also use the following estimate
\begin{equation}\label{tivnormdiff}
  \int_{-\infty}^\infty d\omega\, |\tilde v(k,\omega) - \tilde v(k',\omega)| \leqslant
  \frac{C |k-k'|}{1+k^2}, \quad k\cdot k' \geqslant 0, 
  \quad |k| \leqslant |k'|\leqslant |k|+1.  
\end{equation}
To prove~\re{tivnormdiff} we suppose
$0 \leqslant k \leqslant k' \leqslant k+1$.
We have
\begin{align*}
  &\int_{-\infty}^\infty d\omega\, |\tilde v(k,\omega) - \tilde v(k',\omega)| = 
  \frac{1}{\pi} \int_{k}^{k'} d\omega\, \omega
  \left|\hat u_0\left(k, \sqrt{\omega^2-k^2}\right)\right| +\\
  +&\frac{1}{\pi} \int_{k'}^\infty d\omega\, \omega
  \left|\hat u_0\left(k, \sqrt{\omega^2-k^2}\right) - 
  \hat u_0\left(k', \sqrt{\omega^2-k'^2}\right)\right|.
\end{align*}
The first integral is estimated by the r.h.s. of~\re{tivnormdiff}
in view of~\re{u0decay}.
In the second integral we make change of variable
$l = \sqrt{\omega^2-k'^2}$:
\begin{align}
  &\int_{0}^\infty dl\, l\,
  \left|\hat u_0\left(k, \sqrt{l^2+k'^2-k^2}\right) - 
  \hat u_0\left(k', l\right)\right| \leqslant \notag\\
  \leqslant&\int_{0}^\infty dl\, l\,
  \left|\hat u_0\left(k, \sqrt{l^2+k'^2-k^2}\right) - \hat u_0(k, l)\right| +
  \int_{0}^\infty dl\, l\,
  \left|\hat u_0(k, l) - \hat u_0(k', l)\right|. \label{sqr}
\end{align}
Now we estimate the second integral in the obtained expression
using the estimate of
$\partial\hat u_0 / \partial k$ in~\re{u0decay}:
\begin{align*}
  \int_{0}^\infty dl\, l\,
  \left|\hat u_0(k, l) - \hat u_0(k', l)\right| \leqslant
  \int_{0}^\infty dl\, l\,
  (k'-k)\, \frac{C}{(1+k^3)(1+l^3)} \leqslant
  \frac{C (k'-k)}{1+k^3}.
\end{align*}
The first integral in the r.h.s. of~\re{sqr} can be estimated with the help of the estimate of
$\partial\hat u_0 / \partial l$ in~\re{u0decay}:
\begin{align*}
  &\int_{0}^\infty dl\, l\,
  \left|\hat u_0\left(k, \sqrt{l^2+k'^2-k^2}\right) - \hat u_0(k, l)\right| \leqslant\\
  \leqslant&\int_{0}^\infty dl\, l\,
  \frac{C}{(1+k^3)(1+l^3)} \left(\sqrt{l^2+k'^2-k^2} - l\right) \leqslant\\
  \leqslant&\frac{C (k'^2-k^2)}{1+k^3}
  \int_{0}^\infty dl\, \frac{l}{(1+l^3)(\sqrt{l^2+k'^2-k^2} + l)}.
\end{align*}
Here the integrand can be estimated by $1/(1+l^3)$.
The factor before the integral is estimated by the r.h.s. of~\re{tivnormdiff}.

\section{Another representation of the kernel $K_h$}
Rewrite~\re{Kh} in the following form
\begin{align*}
  &K_h(x, y_0, t) = \\
  &=\frac{1}{2\pi} \, 
  \sum_\pm\int_0^{\pi/2} ds\, \frac{1}{2\sqrt{\pi h(c\pm ix)}}\,
   \exp\left(-\frac{1}{4h(c \pm ix)}\left(x \pm i\sqrt{y_0^2-t^2}\, {\sin}s\right)^2\right).
\end{align*}
The integrand is equal to
\[
  \frac{1}{2\pi} \int_{-\infty}^\infty dk\, e^{-ikx - hk^2(c\pm ix) \pm k\,\sqrt{y_0^2-t^2}\, {\sin}s}
\]
(the integral is absolutely convergent, since $c,h > 0$).
This yields
\begin{equation}\label{KH}
  K_h(x, y_0, t) = \frac{1}{4\pi} \, 
  \sum_\pm\int_{-\infty}^\infty dk\, e^{-ikx - hk^2(c\pm ix)}
  H\left(\pm k\, \sqrt{y_0^2-t^2}\right),
\end{equation}
where
\[
  H(z) := \frac{1}{\pi} \int_0^{\pi/2} ds\, e^{z\,{\sin}s}.
\]

Define $G_\pm(k,\omega)$ as the inverse Fourier transform in $t$ of the function
\[
  \theta(y_0-|t|)\, H\left(\pm k\,\sqrt{y_0^2-t^2}\right).
\]
We do not indicate the dependence of $G_\pm$ on $y_0$ explicitly.
To study the functions $G_\pm$ we need the following relations:
\begin{equation}\label{HJ}
  H(z) + H(-z) = J_0(i z)
\end{equation}
($J_0$ is the Bessel function), 
\begin{equation}\label{besselj}
  \frac{1}{2\pi} \int_{-y_0}^{y_0} dt\, e^{i\omega t}\,
  \frac{1}{2} J_0\(ik\sqrt{y_0^2-t^2}\) =
  \frac{\sin\(y_0 \sqrt{\omega^2-k^2}\)}{2\pi\sqrt{\omega^2-k^2}}.
\end{equation}
The equality~\re{HJ} follows from the definition of $H$
and the following representation for the Bessel function~\cite{wats}
\[
  J_0(z) = \frac{1}{2\pi} \int_{-\pi}^\pi ds\, e^{i z\, {\sin}s}.
\]
The equality~\re{besselj} is just the formula for the inverse Fourier transform in $t$ of the function
\[
  \Phi(y_0,t) := \frac{1}{2} \theta(y_0-|t|) J_0\(ik\sqrt{y_0^2-t^2}\).
\]
The function $\Phi$ is the solution of the following Cauchy problem
\begin{align*}
  \diffd{\Phi}{y_0} - \diffd{\Phi}{t} - k^2\Phi = 0, \quad
  \Phi|_{y_0=0} = 0, \quad \diff{\Phi}{y_0}\bigg|_{y_0=0} = \delta(t),
\end{align*}
where the ``time'' variable is $y_0$.
This can be checked directly or deduced from the results of~\cite[Ch. V]{K}.
Hence the inverse Fourier transform $\widecheck\Phi(y_0,\omega)$
in $t$ of $\Phi(y_0,t)$ satisfies
\begin{align*}
  \diffd{\widecheck\Phi}{y_0} + (\omega^2 - k^2)\widecheck\Phi = 0, \quad
  \widecheck\Phi|_{y_0=0} = 0, \quad \diff{\widecheck\Phi}{y_0}\bigg|_{y_0=0} = \frac{1}{2\pi},
\end{align*}
which implies~\re{besselj}.

\begin{prop}\label{Gprop}
  The following inequalities hold true
  \begin{enumerate}
  \item $|G_\pm(k,\omega)| \leqslant C$, if $\mp k \geqslant 0$; \label{1}
  \item $|G_\pm(k,\omega)| \leqslant C$, if $|k| \leqslant |\omega|$;\label{2}
  \end{enumerate}
  (constant $C$ is independent of $k$, $\omega$).
\end{prop}
\begin{proof}
By the definition of $H$ we have
\[
  G_\pm(k,\omega) = \frac{1}{2\pi^2} \int_{-y_0}^{y_0} dt\, e^{i\omega t}
  \int_0^{\pi/2} ds\, e^{\pm k\,\sqrt{y_0^2-t^2} \,{\sin}s}.
\]
This leads to the first inequality of the Proposition.
Due to~\re{besselj} and \re{HJ} we have
\begin{equation}\label{GG}
  G_+(k,\omega) + G_-(k,\omega) = 
  \frac{\sin(y_0 \sqrt{\omega^2-k^2})}{\pi \sqrt{\omega^2-k^2}}.
\end{equation}
The r.h.s. is bounded if $|k| \leqslant |\omega|$;
together with the first inequality this implies the second inequality of the Proposition.
\end{proof}

\section{``Nonlocal'' version of the formula~\re{mainn}}
In this section we prove the relation
  \begin{equation}\label{lim}
    u(0, y_0, 0) = \lim_{h\to 0+} \int_{-y_0}^{y_0} dt 
    \int_{-\infty}^\infty dx\, K_h(x, y_0, t)\, v(x,t), 
  \end{equation}
which differs from~\re{mainn} in the set of integration
(recall that $v$ was defined by~\re{vdef}).

To prove the equality~\re{lim} we show that the r.h.s. coincides with the r.h.s. of~\re{vsin}.
In view of~\re{KH} for the integral in $x$ in the r.h.s. of~\re{lim} we have
\begin{align*}
  &\sum_\pm \frac{1}{4\pi} \int_{-\infty}^\infty dx\, v(x,t)
  \int_{-\infty}^\infty dk\, e^{-ikx - hk^2(c\pm ix)} H\left(\pm k\,\sqrt{y_0^2-t^2}\right) =\\
  =&\sum_\pm \frac{1}{4\pi} 
  \int_{-\infty}^\infty dk\, e^{-hc k^2} 
  H\left(\pm k\,\sqrt{y_0^2-t^2}\right)
  \int_{-\infty}^\infty dx\, e^{-ikx \mp hk^2 ix}\, v(x,t) =\\
  =&\sum_\pm \frac{1}{4\pi} 
  \int_{-\infty}^\infty dk\, e^{-hc k^2} 
  H\left(\pm k\,\sqrt{y_0^2-t^2}\right)
  \overline v(k\pm hk^2,t).
\end{align*}
Here $\overline v(k,t)$ is Fourier transform of $v(x,t)$ in $x$.
We can change the order of integration in the second equality,
since $v(\cdot, t)$ is compactly supported.
Now for the integral in $t$ in the r.h.s. of~\re{lim} we can write
\begin{align}
  &\sum_\pm \frac{1}{4\pi} \int_{-y_0}^{y_0} dt
  \int_{-\infty}^\infty dk\, e^{-hc k^2} 
  H\left(\pm k\,\sqrt{y_0^2-t^2}\right)
  \overline v(k\pm hk^2,t) =\notag\\
  =&\sum_\pm \frac{1}{4\pi} \int_{-\infty}^\infty dk\, e^{-hc k^2} 
  \int_{-y_0}^{y_0} dt\, H\left(\pm k\,\sqrt{y_0^2-t^2}\right)
  \overline v(k\pm hk^2,t).\label{go}
\end{align}
The function $\overline v(k,\cdot)$ equals the inverse Fourier transform
of $\tilde v(k, \cdot)$, which belongs to $L_2({\mathbb R})$
(the latter can be proved analogously to~\re{tivnorm}).
Hence $\overline v(k,\cdot)\in L_2({\mathbb R})$ and for the integral in $t$
in the r.h.s. of~\re{go} we have
\begin{align*}
  \int_{-y_0}^{y_0} dt\, H\left(\pm k\,\sqrt{y_0^2-t^2}\right)
  \overline v(k\pm hk^2,t) =
  \int_{-\infty}^\infty d\omega\, G_\pm(k, \omega)\, \tilde v(k\pm hk^2, \omega).
\end{align*}
The expression obtained in~\re{go} can written as follows
\begin{align}
  &\sum_\pm \frac{1}{4\pi} \int_{-\infty}^\infty dk\, e^{-hc k^2} 
  \int_{-\infty}^\infty d\omega\, G_\pm(k, \omega)\, \tilde v(k\pm hk^2, \omega). \label{gv}
\end{align}
We show that this tends to
\[
  \sum_\pm \frac{1}{4\pi} \int_{-\infty}^\infty dk 
  \int_{-\infty}^\infty d\omega\, G_\pm(k, \omega)\, \tilde v(k, \omega)
\]
as $h\to 0$, which is equal to~\re{vsin} in view of~\re{GG}.

To calculate the limit of~\re{gv} we estimate the integral
\begin{align}
  \int_{-\infty}^\infty dk\, 
  \int_{-\infty}^\infty d\omega\, |e^{-hc k^2} \tilde v(k \pm hk^2, \omega) - \tilde v(k, \omega)| \cdot
  |G_\pm(k, \omega)|.\label{limdiff}
\end{align}
Consider the case ``$-$'' (the other case can be considered analogously).
First we inspect the integral over the set $0\leqslant k<\infty$, $\omega\in{\mathbb R}$.
The function $G_-(k,\omega)$ can be estimated by the constant $C$ due to the inequality (1) of Proposition~\ref{Gprop}.
So we need to estimate the integral
\begin{align}
  &\int_0^\infty dk\, 
  \int_{-\infty}^\infty d\omega\, |e^{-hc k^2} \tilde v(k - hk^2, \omega) - \tilde v(k, \omega)| 
  = \notag\\
  =&\int_0^{h^{-\gamma}} + \int_{h^{-\gamma}}^\infty dk \int_{-\infty}^\infty d\omega\, 
  |e^{-hc k^2} \tilde v(k - hk^2, \omega) - \tilde v(k, \omega)|,
  \label{limdiff222}
\end{align}
where $0<\gamma<1/2$.
For the first integral in the r.h.s. we have
\begin{align*}
  &\int_0^{h^{-\gamma}} dk\, 
  \int_{-\infty}^\infty d\omega\, |e^{-hc k^2} \tilde v(k - hk^2, \omega) - \tilde v(k, \omega)| 
  \leqslant \notag\\
  \leqslant&\int_0^{h^{-\gamma}} dk \int_{-\infty}^\infty d\omega\, 
  [e^{-hc k^2} |\tilde v(k - hk^2, \omega) - \tilde v(k, \omega)| + (1-e^{-hc k^2})|\tilde v(k, \omega)|].
\end{align*}
If $0\leqslant k<h^{-\gamma}$ then $hck^2 \leqslant c h^{1-2\gamma}$, so
$1-e^{-hc k^2} \leqslant C h^{1-2\gamma}$. 
Combining this with~\re{tivnorm}, we obtain
\begin{align*}
  \int_0^{h^{-\gamma}} dk \int_{-\infty}^\infty d\omega\, 
  (1-e^{-hc k^2})|\tilde v(k, \omega)| \leqslant C h^{1-2\gamma} \to 0, \quad h\to 0.
\end{align*}
Next in view of~\re{tivnormdiff} we have
\begin{align*}
  \int_0^{h^{-\gamma}} dk \int_{-\infty}^\infty d\omega\, 
  |\tilde v(k - hk^2, \omega) - \tilde v(k, \omega)| \leqslant
  \int_0^{h^{-\gamma}} dk\, \frac{C\cdot hk^2}{1+(k-hk^2)^2}.
\end{align*}
The inequality~\re{tivnormdiff} is applicable if $h\leqslant 1$,
since on the set of integration we have
$hk^2 \leqslant h^{1-2\gamma} \leqslant 1$.
Besides, for sufficiently small $h$ the inequality
$k-hk^2 \geqslant k/2$ holds true,
hence the integral obtained above
can be estimated by
\[
  C h^{1-2\gamma} \int_0^{h^{-\gamma}} \frac{dk}{1+k^2} \to 0, \quad h\to 0.
\]
Next the second integral in the r.h.s. of~\re{limdiff222} is majorized by
\begin{align*}
  \int_{h^{-\gamma}}^\infty dk \int_{-\infty}^\infty d\omega\, 
  [e^{-hc k^2} |\tilde v(k - hk^2, \omega)| + |\tilde v(k, \omega)|].
\end{align*}
Here the integral of $|\tilde v(k, \omega)|$ tends to zero as $h\to 0$ due to~\re{tivnorm}.
The integral of $e^{-hc k^2} |\tilde v(k - hk^2, \omega)|$ equals the sum of two integrals over the intervals
$h^{-\gamma}<k<h^{-\beta}$ and $h^{-\beta}<k<\infty$, where $1/2<\beta<1$.
In case $h^{-\gamma}<k<h^{-\beta}$ for sufficiently small $h$
we use the inequality $k-hk^2 \geqslant k/2$ (which follows from $\beta<1$) and the inequality~\re{tivnorm}:
\begin{align*}
  \int_{h^{-\gamma}}^{h^{-\beta}} dk \int_{-\infty}^\infty d\omega\, 
  |\tilde v(k - hk^2, \omega)| \leqslant
  \int_{h^{-\gamma}}^{h^{-\beta}} dk\, \frac{C}{1+k^2} \to 0, \quad h\to 0.
\end{align*}
For $h^{-\beta}<k<\infty$ we simply estimate the integral of
$\tilde v$ in $\omega$ by constant in accordance with~\re{tivnorm}:
\begin{align*}
  &\int_{h^{-\beta}}^\infty dk \int_{-\infty}^\infty d\omega\, 
  e^{-hc k^2} |\tilde v(k - hk^2, \omega)| \leqslant
  C \int_{h^{-\beta}}^\infty dk\, e^{-hc k^2} = \\
  &=C h^{-\beta} \int_1^\infty ds\, e^{-h^{1-2\beta} c s^2} \leqslant
  C h^{-\beta} \int_1^\infty ds\, e^{-h^{1-2\beta} c s} =
  C h^{\beta-1} e^{-h^{1-2\beta}}.
\end{align*}
The obtained majorant tends to zero as $h\to 0$,
since $\beta>1/2$.

To estimate the expression~\re{limdiff}
it remains to consider the corresponding integral over the set
$-\infty< k< 0, \omega\in{\mathbb R}$.
In this case the function $G_-(k,\omega)$ is majorized by constant $C$ as well.
Indeed, the integrand vanishes if $|\omega|<|k|$, which follows
from~\re{tiv}. From the other hand, if $|\omega|>|k|$
the inequality (2) of Proposition~\ref{Gprop} holds true.
Further
\begin{align*}
  &\int_{-\infty}^0 dk\, 
  \int_{-\infty}^\infty d\omega\, |e^{-hc k^2} \tilde v(k - hk^2, \omega) - \tilde v(k, \omega)|
  \leqslant \notag\\
  \leqslant& \int_{-h^{-\gamma}}^0 dk \int_{-\infty}^\infty d\omega\, |e^{-hc k^2} \tilde v(k - hk^2, \omega) - \tilde v(k, \omega)| + \notag\\
  +&\int_{-\infty}^{-h^{-\gamma}} dk \int_{-\infty}^\infty d\omega\, \(|\tilde v(k - hk^2, \omega)| + 
  |\tilde v(k, \omega)| \).\end{align*}
Here $0<\gamma< 1/2$.
The second integral in the r.h.s. tends to zero
as $h\to 0$ in view of the inequality $|k-hk^2| > |k|$
and estimate~\re{tivnorm}.
The first integral can be estimated similarly
to the first integral in the r.h.s. of~\re{limdiff222}.

The relation~\re{lim} is now proved.

\section{Derivation of formula~\re{mainn}}\label{formula_derive}
In view of~\re{lim} to prove formula~\re{mainn} it remains
to show that
\begin{equation}\label{limnull}
  \lim_{h\to 0+} \int_{-y_0}^{y_0} dt
  \int_{|x| > D\(\sqrt{y_0^2-t^2}\)+\varepsilon} dx\,
  K_h(x, y_0, t)\, v(x,t) = 0.
\end{equation}
Recall that we suppose $u_0$ to be compactly supported.
This means that for some $d$ the inequalities
$-y_0\leqslant t\leqslant y_0$, $x^2+y^2 > d^2$ imply that $u(x,y,t)=0$.
Therefore, if $-y_0\leqslant t\leqslant y_0$,
$|x| > d$, we have $v(x,t)=0$,
and thus to estimate the integral in~\re{limnull}
we need to estimate $K_h(x, y_0, t)$ on the set
\begin{equation}\label{xt}
  D\(\sqrt{y_0^2-t^2}\)+\varepsilon \leqslant |x| \leqslant d.
\end{equation}
We prove the following inequality
\begin{equation}\label{Khloc}
  |K_h(x,y_0,t)| \leqslant C h^{-1/2} e^{-a\varepsilon^2/h},
\end{equation}
where $x,t$ satisfy~\re{xt}, $0<\varepsilon\leqslant d$
(if $\varepsilon>d$, then the set~\re{xt} is empty), $h>0$, and
$C,a$ are positive constants independent of
$x$, $t$, $h$, $\varepsilon$.

Denote by $F$ the exponent in the integral in~\re{Kh}.
Also put $z:=\sqrt{y_0^2-t^2}$, $\sigma:={\sin}s$.
We have
\begin{align*}
  {\rm Re} F = \frac{c (-x^2 + z^2\sigma^2) - 2x^2 z\sigma}{4 h (c^2+x^2)}.
\end{align*}
The function ${\rm Re} F$ is convex in $\sigma$, so it satisfies
the inequality
\[
  {\rm Re} F \leqslant \max({\rm Re} F|_{\sigma=0}, {\rm Re} F|_{\sigma=1})
\]
on the interval $0\leqslant \sigma\leqslant 1$.
Since $|x|\geqslant\varepsilon$, we have
\[
  {\rm Re} F|_{\sigma=0} = \frac{-cx^2}{4h(c^2+x^2)} \leqslant
  \frac{-c\varepsilon^2}{4h(c^2+\varepsilon^2)} \leqslant \frac{-c\varepsilon^2}{4h(c^2+d^2)}.
\]
Next
\begin{align*}
  &c z^2 - x^2 (2z+c) = (2z+c)(D(z)^2-x^2) \leqslant
  -c(x^2-D(z)^2) <\\
  &<-c\varepsilon(|x|+D(z)) \leqslant -c\varepsilon^2.
\end{align*}
Therefore,
\begin{align*}
  {\rm Re} F|_{\sigma=1} \leqslant \frac{-c\varepsilon^2}{4h(c^2+d^2)}.
\end{align*}
We proved that if $0\leqslant \sigma\leqslant 1$, then
\[
  {\rm Re} F \leqslant \frac{-a\varepsilon^2}{h}, \quad a = \frac{c}{4(c^2+d^2)}.
\]
This implies~\re{Khloc}, and the relation~\re{limnull} now follows.

The relations~\re{lim} and \re{Khloc} lead to~\re{mainn}.

\end{document}